\documentclass[a4paper,12pt]{amsart}
\title[Galois correspondence]{Galois correspondence for Galois coverings on tropical curves}
\author{Song JuAe}
\address{Tokyo Metropolitan University, 1-1 Minami-Ohsawa, Hachioji, Tokyo, 192-0397, Japan.}
\email{song-juae@ed.tmu.ac.jp}


\usepackage{amsmath,amssymb,amscd}
\usepackage{amsthm}
\usepackage{color}
\usepackage{subfigure}
\usepackage{comment}
\usepackage[dvips]{graphicx}

\newtheorem{dfn}{Definition}[section]
\newtheorem{thm}[dfn]{Theorem}
\newtheorem{prop}[dfn]{Proposition}

\newtheorem{lemma}[dfn]{Lemma}
\newtheorem{rem}[dfn]{Remark}
\newtheorem{ex}[dfn]{Example}

\def\Gamma{\varGamma}

\begin{document}

\maketitle

\begin{abstract}
We define Galois coverings on tropical curves for which a Galois correspondence and a universal mapping property hold.
\end{abstract}

\section{Introduction}

Let $\varphi : \Gamma \to \Gamma^{\prime}$ be a map between tropical curves and $G$ a finite group.
For an isometric action of $G$ on $\Gamma$, $\varphi$ is \textit{Galois} if (1) $\varphi$ is a finite harmonic morphism, (2) the order of $G$ coincides with the degree of $\varphi$, (3) the action of $G$ on $\Gamma$ induces a transitive action on every fiber, and (4) every stabilizer subgroup of $G$ with respect to all but a finite number of points is trivial.
We also say $\varphi$ to be \textit{$G$-Galois}.
Here, a tropical curve is a metric graph that may have edges of length $\infty$, and a finite harmonic morphism is defined as a morphism of our category of tropical curves (see Section \ref{section2} for more details).
When $\varphi$ is $G$-Galois, for a finite harmonic morphism $\psi : \Gamma \to \Gamma^{\prime \prime}$, we define $G(\psi)$ as the subset of $G$ consisting of all elements $g$ of $G$ satisfying $\psi \circ g = \psi$.
For a fixed $G$-Galois covering $\varphi : \Gamma \to \Gamma^{\prime}$, let $A^{\prime}$ be the set of all finite haromonic morphisms $\widetilde{\psi} : \Gamma \to \widetilde{\Gamma^{\prime \prime}}$ each of which is $G(\widetilde{\psi})$-Galois.
For two elements $\psi_1$ and $\psi_2$ of $A^{\prime}$, we write $\psi_1 \sim \psi_2$ if there exists a finite harmonic morphism of degree one $\psi_{12}$ satisfying $\psi_1 = \psi_{12} \circ \psi_2$.
This $\sim$ becomes an equivalence relation.
Let $A := A^{\prime} /\!\! \sim$.
For $[\psi_i] \in A$, where $[\psi_i]$ denotes the equivalence class of $\psi_i$, we write $[\psi_1] \le_A [\psi_2]$ if for any $\psi^{\prime}_1 \in [\psi_1]$ and $\psi^{\prime}_2 \in [\psi_2]$, there exists a finite harmonic morphism $\theta$ satisfying $\psi^{\prime}_1 = \theta \circ \psi^{\prime}_2$.
This $\le_A$ is well-defined and becomes a partial order.
Let $B$ be the set of all subgroups of $G$.
For $G_i \in B$, we define a partial order $\le_B$ such that $G_1 \le_B G_2$ if $G_1 \subset G_2$.
Then, we have the following:

\begin{thm}[Galois correspondence for Galois coverings]
	\label{thm:Galois correspondence1}
In the above setting, there exists a one-to-one correspondence between $(A, \le_A)$ and $(B, \le_B)$ reversing (partial) orders.
\end{thm}

For Galois coverings, the following two theorems hold.
Here, for a map between tropical curves $\varphi : \Gamma \to \Gamma^{\prime}$ and an isometric action of a finite group $H$ on $\Gamma$, $\varphi$ is \textit{$H$-normal} if $\varphi$ satisfies the conditions (1), (3), and (4) above.

\begin{thm}
	\label{thm:intermediate1}
Let $\varphi : \Gamma \to \Gamma^{\prime}$ be a $G$-Galois covering, $\psi : \Gamma \to \Gamma^{\prime \prime}$ in $A^{\prime}$, and $\theta : \Gamma^{\prime \prime} \to \Gamma^{\prime}$ the unique continuous map satisfying $\varphi = \theta \circ \psi$.
Then, the following are equivelent:

$(1)$ $G(\psi)$ is a normal subgroup of $G$,

$(2)$ there exists a finite group $H$ isometrically acting on $\Gamma^{\prime \prime}$ such that $\theta$ is $H$-Galois, and

$(3)$ there exists a finite group $H^{\prime}$ isometrically acting on $\Gamma^{\prime \prime}$ such that $\theta$ is $H^{\prime}$-normal.

Moreover, suppose in addition to $(2)$ that for any $g \in G$, there exists $h \in H$ satisfying $\psi \circ g = h \circ \psi$, then $H$ is isomorphic to the quotient group $G / G(\psi)$.
\end{thm}

\begin{thm}[Universal mapping property]
	\label{thm:UMP}
Let $\varphi : \Gamma \to \Gamma^{\prime}$ be $G$-Galois.
Then, for any finite harmonic morphism $\psi : \Gamma \to \Gamma^{\prime \prime}$ satisfying for any $g \in G$, $\psi \circ g = \psi$, there exists a unique finite harmonic morphism $\theta$ satisfying $\psi = \theta \circ \varphi$.
\end{thm}

Our definition of Galois coverings on tropical curves is an enhanced version of the definition in \cite{JuAe2} (the condition (4) is added).
By adding the condition (4), Theorems \ref{thm:Galois correspondence1} and \ref{thm:UMP} above and Theorem \ref{thm:Galoisextension} below, which will be shown in \cite{JuAe3}, hold; see Remark \ref{rem:UMP} and Examples \ref{ex:C}, \ref{ex:C2}, \ref{ex:C1}.

\begin{thm}[\cite{JuAe3}]
	\label{thm:Galoisextension}
Let $\varphi : \Gamma \to \Gamma^{\prime}$ be a finite harmonic morphism between tropical curves and $G$ a finite group isometrically acting on $\Gamma$.
Then, $\varphi$ is $G$-Galois if and only if the action of $G$ on $\operatorname{Rat}(\Gamma)$ naturally induced by the action of $G$ on $\Gamma$ is Galois for the semifield extension $\operatorname{Rat}(\Gamma) / \varphi^{\ast}(\operatorname{Rat}(\Gamma^{\prime}))$.
\end{thm}

Here, $\operatorname{Rat}(\Gamma)$ denotes the semifield of all rational functions on $\Gamma$ and $\varphi^{\ast}(\operatorname{Rat}(\Gamma^{\prime}))$ the pull-back.
In the setting of Theorem \ref{thm:Galoisextension}, ``the action of $G$ on $\operatorname{Rat}(\Gamma)$ is Galois for $\operatorname{Rat}(\Gamma) / \varphi^{\ast}(\operatorname{Rat}(\Gamma^{\prime}))$" means that the $G$-invariant subsemifield of $\operatorname{Rat}(\Gamma)$ is $\varphi^{\ast}(\operatorname{Rat}(\Gamma^{\prime}))$ and that subgroups of $G$ whose invariant semifields coincide are equal.

This paper is organized as follows.
In Section 2, we give the definitions of tropical curves and finite harmonic morphisms between tropical curves.
Section 3 gives proofs of Theorems \ref{thm:Galois correspondence1}, \ref{thm:intermediate1}, and \ref{thm:UMP}.

\section*{Acknowledgements}
The author thanks my supervisor Masanori Kobayashi, Yuki Kageyama, and Yasuhito Nakajima for helpful comments.
This work was supported by JSPS KAKENHI Grant Number 20J11910.

\section{Preliminaries}
	\label{section2}

In this section, we prepare basic definitions related to tropical curves which we need later.
See, for example, \cite{JuAe1} for details.

In this paper, a \textit{graph} is an unweighted, undirected, finite, connected nonempty multigraph that may have loops.
For a graph $G$, the set of vertices is denoted by $V(G)$ and the set of edges by $E(G)$.
The \textit{degree} of a vertex is the number of edges incident to it.
Here, a loop is counted twice.
A \textit{leaf end} is a vertex of degree one.
A \textit{leaf edge} is an edge incident to a leaf end.
A \textit{tropical curve} is the underlying topological space of the pair $(G, l)$ of a graph $G$ and a function $l: E(G) \to {\boldsymbol{R}}_{>0} \cup \{\infty\}$ called a \textit{length function}, where $l$ can take the value $\infty$ on only leaf edges, together with an identification of each edge $e$ of $G$ with the closed interval $[0, l(e)]$.
When $l(e)=\infty$, the interval $[0, \infty]$ is the one point compactification of the interval $[0, \infty)$ and the leaf end of $e$ must be identified with $\infty$.
We regard $[0, \infty]$ not just as a topological space but as almost a metric space.
The distance between $\infty$ and any other point is infinite.
If $E(G) = \{ e \}$ and $l(e)=\infty$, then we can identify either leaf ends of $e$ with $\infty$.
When a tropical curve $\Gamma$ is obtained from $(G, l)$, the pair $(G, l)$ is called a \textit{model} for $\Gamma$.
There are many possible models for $\Gamma$.
A model $(G, l)$ is \textit{loopless} if $G$ is loopless.
Let $\Gamma_{\infty}$ denote the set of all points of $\Gamma$ identified with $\infty$.
An element of $\Gamma_{\infty}$ is called a \textit{point at infinity}.
We frequently identify a vertex (resp. an edge) of $G$ with the corresponding point (resp. the corresponding closed subset) of $\Gamma$.
The \textit{relative interior} $e^{\circ}$ of an edge $e$ is $e \setminus \{v, w\}$ with the endpoint(s) $v, w$ of $e$.

Let $\varphi : \Gamma \to \Gamma^{\prime}$ be a continuous map between tropical curves.
$\varphi$ is a {\it finite morphism} if there exist loopless models $(G, l)$ and $(G^{\prime}, l^{\prime})$ for $\Gamma$ and $\Gamma^{\prime}$, respectively, such that $(1)$ $\varphi(V(G)) \subset V(G^{\prime})$ holds, $(2)$ $\varphi(E(G)) \subset E(G^{\prime})$ holds, and $(3)$ for any edge $e$ of $G$, there exists a positive integer $\operatorname{deg}_e(\varphi)$ such that for any points $x, y$ of $e$, $\operatorname{dist}(\varphi (x), \varphi (y)) = \operatorname{deg}_e(\varphi) \cdot \operatorname{dist}(x, y)$ holds.
Here, $\operatorname{dist}(x, y)$ denotes the distance between $x$ and $y$.
Assume that $\varphi$ is a finite morphism.
$\varphi$ is \textit{harmonic} if for every vertex $v$ of $G$, the sum $\sum_{e \in E(G):\, e \mapsto e^{\prime},\, v \in e} \operatorname{deg}_e(\varphi)$ is independent of the choice of $e^{\prime} \in E(G^{\prime})$ incident to $\varphi(v)$.
This sum is denoted by $\operatorname{deg}_v(\varphi)$.
Then, the sum $\Sigma_{v \in V(G):\, v \mapsto v^{\prime}} \operatorname{deg}_v(\varphi)$ is independent of the choice of vertex $v^{\prime}$ of $G^{\prime}$.
It is said the {\it degree} of $\varphi$ and written by $\operatorname{deg}(\varphi)$.
If both $\Gamma$ and $\Gamma^{\prime}$ are singletons, we regard $\varphi$ as a finite harmonic morphism that can have any number as its degree.
Note that if $\varphi \circ \psi$ is a composition of finite harmonic morphisms, then it is also a finite harmonic morphism of degree $\operatorname{deg}(\varphi) \cdot \operatorname{deg}(\psi)$, and thus tropical curves and finite harmonic morphisms between them make a category.

\begin{rem}
	\label{rem:isom}
\upshape{
Let $\varphi : \Gamma \to \Gamma^{\prime}$ be a map between tropical curves.
Then $\varphi$ is a continuous map whose restriction on $\Gamma \setminus \Gamma_{\infty}$ is an isometry if and only if it is a finite harmonic morphism of degree one.
In this paper, we will use the word ``a finite group $G$ isometrically acts on a tropical curve $\Gamma$" as the meaning that $G$ continuously acts on $\Gamma$ and it is isometric on $\Gamma \setminus \Gamma_{\infty}$.
}
\end{rem}

\begin{rem}
	\label{rem:pi1}
\upshape{
Let $\Gamma$ be a tropical curve and $G$ a finite group isometrically acting on $\Gamma$.
Let $\Gamma / G$ be the quotient space (as a topological space) and $\pi_G : \Gamma \to \Gamma / G$ be the natural surjection.
Fix a loopless model $(V, E, l)$ ($V$ is a set of vertices and $E$ is a set of edges) for $\Gamma$ such that for any $g \in G$ and $e \in E$, $g(V) = V$ holds and if $g(e) = e$, then the restriction of $g$ on $e$ is the identity map of $e$.
Note that by the proof of \cite[Lemma 3.4]{JuAe1}, we can choose such a loopless model.
Let $V^{\prime} := \pi_G(V)$, $E^{\prime} := \pi_G(E)$ and for any $e \in E$, $l^{\prime}(\pi_G(e)) := |G_e| \cdot l(e)$, where $G_e$ denotes the stabilizer subgroup of $G$ with respect to $e$ and $|G_e|$ the order of $G_e$.
Then, $(V^{\prime}, E^{\prime}, l^{\prime})$ gives $\Gamma / G$ a tropical curve structure and is a loopless model for the quotient tropical curve $\Gamma / G$.
By loopless models $(V, E, l)$ and $(V^{\prime}, E^{\prime}, l^{\prime})$ for $\Gamma$ and $\Gamma / G$, respectively, $\pi_G$ is a finite harmonic morphism of degree $|G|$.
}
\end{rem}

\section{Main results}

In this section, we give our definition of Galois coverings on tropical curves and proofs of Theorems \ref{thm:Galois correspondence1}, \ref{thm:intermediate1}, and \ref{thm:UMP}.

\begin{dfn}
	\label{dfn:Galois1}
\upshape{
Let $\Gamma$ be a tropical curve and $G$ a finite group.
An isometric action of $G$ on $\Gamma$ is \textit{Galois} if there exists a finite subset $U^{\prime}$ of $\Gamma / G$ such that for any $x^{\prime} \in (\Gamma / G) \setminus U^{\prime}$, $|\pi^{-1}_G(x^{\prime})| = |G|$ holds.
}
\end{dfn}

\begin{dfn}
	\label{dfn:Galois2}
\upshape{
Let $\varphi : \Gamma \to \Gamma^{\prime}$ be a map between tropical curves.
$\varphi$ is \textit{Galois} if there exists a Galois action of a finite group $G$ on $\Gamma$ such that there exists a finite haramonic morphism of degree one $\theta : \Gamma / G \to \Gamma^{\prime}$ satisfying $\varphi \circ g = \theta \circ \pi_G$ for any $g \in G$.
Then, we say that $\varphi$ is a \textit{$G$-Galois covering} on $\Gamma^{\prime}$ or just \textit{$G$-Galois}.
}
\end{dfn}

\begin{proof}[Proof of Theorem \ref{thm:Galois correspondence1}]

Let 
\begin{align*}
\Phi : A \to B; \qquad [\psi] \mapsto G(\psi)
\end{align*}
and
\begin{align*}
\Psi : B \to A; \qquad G^{\prime} \mapsto [\pi_{G^{\prime}}].
\end{align*}

We shall show that if $\psi_1 \sim \psi_2$ holds, then $G(\psi_1) = G(\psi_2)$ holds.
For any $f \in G(\psi_2)$, we have
\begin{align*}
\psi_1 \circ f = (\psi_{12} \circ \psi_2) \circ f = \psi_{12} \circ (\psi_2 \circ f) = \psi_{12} \circ \psi_2 = \psi_1.
\end{align*}
Thus $f \in G(\psi_1)$ holds.
The inverse inclusion can be shown similarly.
Therefore, $\Phi$ is well-defined.

We shall show that for any $G^{\prime} \in B$, the action of $G^{\prime}$ on $\Gamma$ induced by that of $G$ on $\Gamma$ is Galois.
Let $U^{\prime}$ be the finite subset of $\Gamma / G$ in Definition \ref{dfn:Galois1}.
For any $x \in \Gamma \setminus \pi^{-1}_G(U^{\prime})$, we have $|G_x| = \frac{|G|}{|Gx|} = 1$, where we use the orbit-stabilizer theorem (cf. \cite[Chapter 6]{Clive=Stuart}) and $Gx$ stands for the orbit of $x$ by $G$, and thus $|G^{\prime}_x| = 1$.
Therefore, $\pi_{G^{\prime}}(\pi^{-1}_G(U^{\prime}))$ is finite and for any $x^{\prime} \in (\Gamma / G^{\prime}) \setminus \pi_{G^{\prime}}(\pi^{-1}_G(U^{\prime}))$, we have $|\pi^{-1}_{G^{\prime}}(x^{\prime})| = |G^{\prime}|$.

We shall show that for any $G^{\prime} \in B$, $G^{\prime} = G(\pi_{G^{\prime}})$ holds.
By the definition of $G(\pi_{G^{\prime}})$, $G^{\prime} \subset G(\pi_{G^{\prime}})$ holds.
Let $g \in G(\pi_{G^{\prime}})$.
For any $x \in \Gamma$, since $(\pi_{G^{\prime}} \circ g)(x) = \pi_{G^{\prime}}(x)$, there exists $g^{\prime} \in G^{\prime}$ such that $g(x) = g^{\prime}(x)$.
Let us assume that $x$ is in $\Gamma \setminus \pi^{-1}_G(U^{\prime})$.
Then $g = g^{\prime}$ must hold.
In fact, if $g \not= g^{\prime}$, then $g^{-1} g \not= 1$ and $g^{-1} g^{\prime}(x) = x$.
Thus, $g^{-1} g^{\prime} \in G_x \setminus \{ 1 \}$ and it contradicts that $|\pi^{-1}_G(x)| = |G|$.
By these arguments, we have $[\pi_{G^{\prime}}] \in A$ and $\Phi \circ \Psi = \operatorname{id}_B$, where $\operatorname{id}_B$ denotes the identity map of $B$.

We shall show that $\Psi \circ \Phi = \operatorname{id}_A$ holds.
By the definitions of $A$ and Galois coverings, for any $\psi \in A^{\prime}$, $\psi \sim \pi_{G(\psi)}$ holds, and hence we have 
\[
(\Psi \circ \Phi)([\psi]) = [\pi_{G(\psi)}] = [\psi].
\]

We shall show that if $[\psi_1] \le_A [\psi_2]$ holds, then $G(\psi_1) \ge_B G(\psi_2)$ holds.
For $g \in G(\psi_2)$, we have $\psi_2 \circ g = \psi_2$.
Since $[\psi_1] \le_A [\psi_2]$, there exists a finite harmonic morphism $\theta$ satisfying $\psi_1 = \theta \circ \psi_2$.
Therefore 
\begin{align*}
\psi_1 \circ g = (\theta \circ \psi_2) \circ g = \theta \circ (\psi_2 \circ g) = \theta \circ \psi_2 = \psi_1
\end{align*}
hold.
Hence $g \in G(\psi_1)$.

Finally, we shall show that if $G_1 \le_B G_2$ holds, then $[\pi_{G_1}] \ge_A [\pi_{G_2}]$ holds.
Since $G_1 \subset G_2$, there exists a unique continuous map $\theta$ satisfying $\theta \circ \pi_{G_1} = \pi_{G_2}$.
We shall show $\theta$ is a finite harmonic morphism.
Let $(V, E, l), (V_1, E_1, l_1)$ and $(V_2, E_2, l_2)$ be loopless models for $\Gamma$, $\Gamma / G_1$ and $\Gamma / G_2$, respectively, compatible with $\pi_{G_i}$, i.e., $\pi_{G_i}(V) = V_i$ for $i = 1, 2$.
Let $e \in E$.
Since each $\pi_{G_i}$ is Galois, for any points $x, y$ of $e$, we have
\begin{align*}
\operatorname{dist}(\pi_{G_2}(x), \pi_{G_2}(y)) = \operatorname{dist}(x, y) = \operatorname{dist}(\pi_{G_1}(x), \pi_{G_1}(y)).
\end{align*}
Hence $\theta$ is a finite morphism and $\operatorname{deg}_{\pi_1(e)}(\theta) = 1$.
Next, we shall show that $\theta$ is harmonic.
Fix a vertex $v^{\prime} \in V_1$ and an edge $e^{\prime \prime} \in E_2$ incident to $\theta(v^{\prime})$.
Since the action of $G$ on $\Gamma$ is isometric and $G_1 \subset G_2$, with any $v \in \pi^{-1}_{G_2}(\theta(v^{\prime}))$, we have
\begin{align*}
\sum_{e^{\prime} \in E_1 :\, e^{\prime} \mapsto e^{\prime \prime},\, v^{\prime} \in e^{\prime}} \operatorname{deg}_{e^{\prime}}(\theta) &= \frac{1}{|\theta^{-1}(\theta(v^{\prime}))|} \cdot \sum_{e^{\prime} \in E_1 : e^{\prime} \mapsto e^{\prime \prime}} \operatorname{deg}_{e^{\prime}}(\theta)\\
&= \frac{1}{|\theta^{-1}(\theta(v^{\prime}))|} \cdot |\{ e^{\prime} \in E_1 \,|\, \theta(e^{\prime}) = e^{\prime \prime} \}|
\end{align*}
and
\begin{align*}
|\theta^{-1}(\theta(v^{\prime}))| = \frac{|\pi_{G_2}^{-1}(\theta(v^{\prime}))|}{|G_1 v|} = \frac{|G_2 v|}{|G_1 v|}.
\end{align*}
With any edge $e \in E$ such that $\pi_{G_2}(e) = e^{\prime \prime}$, the ratio $\frac{|G_2 e|}{|G_1 e|} = \frac{|G_2|}{|G_1|}$ coincides with the number of elements of the inverse image of $e^{\prime \prime}$ by $\theta$.
Therefore we have
\begin{align*}
\sum_{e^{\prime} \in E_1 :\, e^{\prime} \mapsto e^{\prime \prime},\, v^{\prime} \in e^{\prime}} \operatorname{deg}_{e^{\prime}}(\theta) = \frac{|G_1 v|}{|G_2 v|} \cdot \frac{|G_2|}{|G_1|} = \frac{|{G_2}_v|}{|{G_1}_v|} \in \boldsymbol{Z}_{>0},
\end{align*}
where we use the orbit-stabilizer theorem at the second equality.
Hence the sum is independent of the choice of $e^{\prime \prime}$ and thus $\theta$ is harmonic.
\end{proof}

Next, we will prove Theorem \ref{thm:intermediate1}.
For this purpose, we define normal coverings on tropical curves:

\begin{dfn}
	\label{dfn:normal}
\upshape{
Let $\varphi : \Gamma \to \Gamma^{\prime}$ be a map between tropical curves.
$\varphi$ is \textit{normal} if there exists a Galois action of a finite group $H$ on $\Gamma$ and a finite harmonic morphism $\theta : \Gamma / H \to \Gamma^{\prime}$ such that for any $h \in H$, $\varphi \circ h = \theta \circ \pi_H$ holds.
Then, we say that $\varphi$ is a \textit{$H$-normal covering} on $\Gamma^{\prime}$ or just \textit{$H$-normal}.
}
\end{dfn}

By definition, if $\varphi$ is $G$-Galois, then it is $G$-normal.
Normal coverings have the following property:

\begin{prop}
	\label{prop:prenormal}
Let $\varphi : \Gamma \to \Gamma^{\prime}$ be a $H$-normal covering on $\Gamma^{\prime}$.
Then, for any finite harmonic morphism $\psi : \Gamma^{\prime \prime} \rightarrow \Gamma$, there exists a loopless model $(G^{\prime \prime}, l^{\prime \prime})$ for $\Gamma^{\prime \prime}$ such that for any $e^{\prime \prime} \in E(G^{\prime \prime})$ and any automorphim $f$ of $\Gamma^{\prime \prime}$, i.e., a finite harmonic morphism of degree one $\Gamma^{\prime \prime} \to \Gamma^{\prime \prime}$, satisfying $\varphi \circ \psi \circ f = \varphi \circ \psi$, there exists $h \in H$ such that $(\psi \circ f)|_{e^{\prime \prime}} = (h \circ \psi)|_{e^{\prime \prime}}$, where $|_{e^{\prime \prime}}$ denotes the restriction of the map on $e^{\prime \prime}$.
\end{prop}

\begin{proof}
By the proof of \cite[Lemma 3.4]{JuAe1}, we can choose $(G^{\prime \prime}, l^{\prime \prime})$ such that
$\psi(V(G^{\prime \prime}))$ (resp. $(\varphi \circ \psi)(V(G^{\prime \prime}))$) induces a loopless model for $\Gamma$ (resp. $\Gamma^{\prime}$) compatible with $\psi$ (resp. $\varphi \circ \psi$) and for any $e^{\prime \prime} \in E(G^{\prime \prime})$ and $h \in H$, if $h(\psi(e^{\prime \prime})) = \psi(e^{\prime \prime})$, then the restriction $h|_{\psi(e^{\prime \prime})}$ is the identity map on $\psi(e^{\prime \prime})$.
For any edge $e^{\prime \prime} \in E(G^{\prime \prime})$ and any point $x^{\prime \prime}$ of $(e^{\prime \prime})^{\circ}$, since $(\varphi \circ \psi \circ f) (x^{\prime \prime}) = (\varphi \circ \psi) (x^{\prime \prime})$, there exists $h \in H$ satisfying $(\psi \circ f) (x^{\prime \prime}) = (h \circ \psi) (x^{\prime \prime})$.
Since the action of $H$ on $\Gamma$ is isometric and $h$ does not invert any edge, $(\psi \circ f)|_{(e^{\prime \prime})^{\circ}} = (h \circ \psi)|_{(e^{\prime \prime})^{\circ}}$ holds.
Since both $\psi \circ f$ and $h \circ \psi$ are continuous, $(\psi \circ f)|_{e^{\prime \prime}} = (h \circ \psi)|_{e^{\prime \prime}}$ holds.
\end{proof}

This proposition is an analogue of the fact in the field theory that for an algebraic extension $L / K$ of fields, it is normal if and only if all embeddings $\sigma$ of $L$ to the algebraic closure of $K$ containing $L$ which is the identity on $K$ satisfy $\sigma(L) = L$.
In the tropical curve case, however, we must consider only an edge instead of the whole tropical curve.

Let $\varphi : \Gamma \rightarrow \Gamma^{\prime}$ be a $G$-Galois covering on $\Gamma^{\prime}$.
For any $\psi : \Gamma \rightarrow \Gamma^{\prime \prime}$ in $A^{\prime}$, since $G(\psi)$ is a subgroup of $G$ and $\Gamma^{\prime}$ and $\Gamma^{\prime \prime}$ have quotient topologies, there exists a unique continuous map $\theta : \Gamma^{\prime \prime} \to \Gamma^{\prime}$ such that $\varphi = \theta \circ \psi$.

\begin{prop}
	\label{prop:Galois correspondence1}
In the above setting, the following are equivalent:

$(1)$ $G(\psi)$ is a normal subgroup of $G$,

$(2)$ there exists a finite group $H$ isometrically acting on $\Gamma^{\prime \prime}$ such that $\theta$ is $H$-Galois and for any $g \in G$, there exists $h \in H$ satisfying $\psi \circ g = h \circ \psi$, and

$(3)$ there exists a finite group $H^{\prime}$ isometrically acting on $\Gamma^{\prime \prime}$ such that $\theta$ is $H^{\prime}$-normal and for any $g \in G$, there exists $h^{\prime} \in H^{\prime}$ satisfying $\psi \circ g = h^{\prime} \circ \psi$.

Moreover, if the above conditions hold, then $H$ in $(2)$ is isomorphic to the quotient group $G / G(\psi)$.
\end{prop}

For a map between tropical curves, we will use the word ``it is a local isometry" as the meaning that it is continuous on the whole domain tropical curve and that it is a local isometry on the domain tropical curve except all points at infinity.

\begin{proof}[Proof of Proposition \ref{prop:Galois correspondence1}]
Clearly (2) implies (3).

Assume that (3) holds.
For $g \in G$ and $h^{\prime} \in H^{\prime}$, if $\psi \circ g = h^{\prime} \circ \psi$, then $\psi \circ g^{-1} = {h^{\prime}}^{-1} \circ \psi$, and thus for any $g^{\prime} \in G(\psi)$, we have
\begin{align*}
\psi \circ(g^{-1} g^{\prime} g) &= (\psi \circ g^{-1}) \circ g^{\prime}g = ({h^{\prime}}^{-1} \circ \psi) \circ g^{\prime}g\\
&= {h^{\prime}}^{-1} \circ (\psi \circ g^{\prime}) \circ g 
= {h^{\prime}}^{-1} \circ \psi \circ g\\
&= {h^{\prime}}^{-1} \circ (\psi \circ g)
= {h^{\prime}}^{-1} h^{\prime} \circ \psi
= \psi.
\end{align*}
This means that $g^{-1} g^{\prime} g \in G(\psi)$.
Thus (1) holds.

Assume that (1) holds.
By the same argument of the last part of the proof of Theorem \ref{thm:Galois correspondence1}, $\theta$ is a finite harmonic morphism of degree $\frac{|G|}{|G(\psi)|}$.
Here, $G(\psi)$ corresponds to $G_1$ and $G$ to $G_2$ in this case.
We define an isomertic action of $G / G(\psi)$ on $\Gamma^{\prime \prime}$ as $[g](x^{\prime \prime}) := \psi(g(x))$ for $[g] \in G / G(\psi)$, $x^{\prime \prime} \in \Gamma^{\prime \prime}$ and any $x \in \psi^{-1}(x^{\prime \prime})$.
It is well-defined.
In fact, if $x_1 \in \psi^{-1}(x^{\prime \prime})$, then there exists $g_1 \in G(\psi)$ such that $g_1 (x_1) = x$.
Since $G(\psi)$ is a normal subgroup of $G$, there exists $g_2 \in G(\psi)$ such that $g(x) = g g_1 (x_1) = g_2 g(x_1)$.
Thus we have $\psi(g(x)) = \psi(g_2g(x_1)) = \psi(g(x_1))$.
If $[g] = [\widetilde{g}]$, then there exists $g_1 \in G(\psi)$ such that $g = g_1 \widetilde{g}$.
Hence $\psi(g(x)) = \psi(g_1 \widetilde{g}(x)) = \psi (\widetilde{g}(x))$.
Since the action of $G$ on $\Gamma$ is Galois, there exists a finite subset $U^{\prime}$ of $\Gamma / G$ such that for any $x^{\prime} \in (\Gamma / G) \setminus U^{\prime}$, $|\pi^{-1}_G(x^{\prime})| = |G|$ holds.
Let $O$ be the finite subset $(\pi_{G / G(\psi)} \circ \psi)(\pi^{-1}_G(U^{\prime}))$ of $\Gamma^{\prime \prime} / (G / G(\psi))$, where $\pi_{G / G(\psi)} : \Gamma^{\prime \prime} \to \Gamma^{\prime \prime} / (G / G(\psi))$ is the natural surjection.
Let $[x^{\prime \prime}] \in (\Gamma^{\prime \prime} / (G / G(\psi))) \setminus O$.
We have $\pi^{-1}_{G / G(\psi)} ([x^{\prime \prime}]) = (G / G(\psi)) x^{\prime \prime}$.
Since $|\varphi^{-1}(\theta(x^{\prime \prime}))| = |G|$ and $\psi$ is $G(\psi)$-Galois, we have
\begin{align*}
|(G / G(\psi))x^{\prime \prime}| = |\psi(\varphi^{-1}(\theta(x^{\prime \prime})))| = \frac{|G|}{|G(\psi)|},
\end{align*}
and thus the action of $G / G(\psi)$ on $\Gamma^{\prime \prime}$ is Galois.
Let 
\begin{align*}
\phi : \Gamma^{\prime \prime} / (G / G(\psi)) \to \Gamma^{\prime}; \qquad \pi_{G / G(\psi)}(x^{\prime \prime}) \mapsto \theta(x^{\prime \prime}),
\end{align*}
where $x^{\prime \prime} \in \Gamma^{\prime \prime}$.
With any fixed element $x \in \psi^{-1}((G / G(\psi))x^{\prime \prime})$, since
\begin{align*}
\theta((G / G(\psi)) x^{\prime \prime}) = \theta(\psi(Gx)) = \theta(\psi(x)) = \varphi(x),
\end{align*}
we have $\theta(x^{\prime \prime}) = \varphi(x)$, and thus $\phi$ is well-defined.
By definition, for any $[g] \in G / G(\psi)$, $\phi \circ \pi_{G / G(\psi)} = \theta \circ [g]$ holds.
Since $\pi_{G / G(\psi)}$ and $\theta$ are local isometries and $[g]$ is an automorphism, $\phi$ is a local isometry.
When $\theta(x^{\prime \prime}) = \theta(y^{\prime \prime})$, with any $x \in \psi^{-1}((G / G(\psi)) x^{\prime \prime})$ and $y \in \psi^{-1}((G / G(\psi)) y^{\prime \prime})$, $(G / G(\psi))x^{\prime \prime} = \psi(Gx) = \psi(Gy) = (G / G(\psi))y^{\prime \prime}$ hold, and thus $\pi_{G / G(\psi)}(x^{\prime \prime}) = \pi_{G / G(\psi)}(y^{\prime \prime})$ holds.
This means that $\phi$ is injective and hence a finite harmonic morphism of degree one.
In conclusion, $\theta$ is $G / G(\psi)$-Galois.

We shall show the last assertion.
By (1), the quotient group $G / G(\psi)$ naturally acts on $\Gamma^{\prime \prime}$ as above.
Since each stabilizer subgroup of $G$ with respect to each point of $\Gamma$ except a finite number of points is trivial, the natural action is faithful, i.e., for any equivalence class $[g] \in G / G(\psi)$ other than $[1]$, there exists a point $x^{\prime \prime} \in \Gamma^{\prime \prime}$ such that $[g](x^{\prime \prime}) \not= x^{\prime \prime}$.
For any $g \in G$, $\psi \circ g = [g] \circ \psi$ and there exists $h \in H$ such that $\psi \circ g = h \circ \psi$.
Since the natural action above is faithful, if $[g] \not= [\widetilde{g}]$, then $[g] \circ \psi \not= [\widetilde{g}] \circ \psi$.
Thus, the map 
\begin{align*}
G / G(\psi) \to H; \qquad [g] \mapsto h
\end{align*}
is injective.
Since\begin{align*}
|H| = \operatorname{deg}(\theta) = \frac{|G|}{|G(\psi)|} = |G / G(\psi)|
\end{align*}
holds, this map is surjective.
By definition, it is a group homomorphism.
In conclusion, we have the desired group isomorphism.
\end{proof}

To remove the last conditions in (2) and (3) of Proposition \ref{prop:Galois correspondence1}, we peove the following lemma by using Proposition \ref{prop:prenormal}:

\begin{lemma}
	\label{lem:prenormal}
In the same setting in Proposition \ref{prop:Galois correspondence1}, the following are equivalent:

$(1)$ there exists a finite group $H$ isometrically acting on $\Gamma^{\prime \prime}$ such that $\theta$ is $H$-normal, and

$(2)$ there exists a finite group $H^{\prime}$ isometrically acting on $\Gamma^{\prime \prime}$ such that $\theta$ is $H^{\prime}$-normal and for any $g \in G$, there exists $h^{\prime} \in H^{\prime}$ satisfying $\psi \circ g = h^{\prime} \circ \psi$.
\end{lemma}

\begin{proof}
(2) clearly implies (1).

Assume that (1) holds.
For any $g \in G$, we define $h^{\prime}_g : \Gamma^{\prime \prime} \to \Gamma^{\prime \prime}; \psi(x) \mapsto (\psi \circ g)(x)$, where $x \in \Gamma$.
Then, by Proposition \ref{prop:prenormal}, $h^{\prime}_g$ is a local isometry.
Since $\psi \circ g$ is surjective and continuous, $h^{\prime}_g$ is an automorphism of $\Gamma^{\prime \prime}$.
Set $H^{\prime}$ as the group generated by these $h^{\prime}_g$.
Then $H^{\prime}$ induces an isometric and faithful action on $\Gamma^{\prime}$.
For any $x \in \Gamma$, we check that $H^{\prime} \psi(x) = H \psi(x)$.
Since for any $g \in G$, there exists $h \in H$ such that 
\begin{align*}
(h^{\prime}_g \circ \psi)(x) = (\psi \circ g) (x) = (h \circ \psi)(x),
\end{align*}
we have $H^{\prime} \psi(x) \subset H \psi(x)$.
Conversely, for any $h \in H$, since
\begin{align*}
\varphi(\psi^{-1} ((h \circ \psi) (x)) ) = \theta ((h \circ \psi) (x)) = (\theta \circ \psi) (x) = \varphi (x),
\end{align*}
$\psi^{-1} ((h \circ \psi) (x))$ is contained in $Gx$.
Therefore, for any $\widetilde{x} \in \psi^{-1}((h \circ \psi)(x))$, there exists $g \in G$ such that $\widetilde{x} = g(x)$.
Thus, we have
\begin{align*}
(h \circ \psi)(x) = \psi(\widetilde{x}) = \psi(g(x)) = (\psi \circ g)(x) = (h^{\prime}_g \circ \psi) (x),
\end{align*}
and hence $H^{\prime} \psi(x) \supset H \psi(x)$.
From this, the action of $H^{\prime}$ on $\Gamma^{\prime}$ induces a transitive action on each fiber by the transitivity of the action of $H$ on each fiber.
In conclusion, $\theta$ is $H^{\prime}$-normal and satisfies that $\theta \circ \psi = \varphi$ and for any $g \in G$, there exists $h^{\prime}_g \in H^{\prime}$ such that $\psi \circ g = h^{\prime}_g \circ \psi$.
\end{proof}

By Proposition \ref{prop:Galois correspondence1} and Lemma \ref{lem:prenormal}, we prove theorem \ref{thm:intermediate1}.

Finally, we shall prove Theorem \ref{thm:UMP}.
This theorem means that all Galois coverings are categorical quotients.

\begin{proof}[Proof of Theorem \ref{thm:UMP}]
Since $\Gamma^{\prime}$ has the quotient topology, we have a unique continuous map $\theta$ satisfying $\psi = \theta \circ \varphi$, and thus it is enough to check that $\theta$ is a finite harmonic morphism.
Fix loopless models $(V, E, l), (V^{\prime}, E^{\prime}, l^{\prime})$ and $(V^{\prime \prime}, E^{\prime \prime}, l^{\prime \prime})$ for $\Gamma, \Gamma^{\prime}$ and $\Gamma^{\prime \prime}$, respectively, compatible with $\varphi$ and $\psi$.
We can choose such loopless models by the assumptions.
Let $e \in E$.
As for any $x, y \in e$, $\operatorname{dist}(x, y) = \operatorname{dist}(\varphi(x), \varphi(y))$ holds, we have $\operatorname{deg}_e(\psi) = \operatorname{deg}_{\varphi(e)}(\theta)$.
Therefore, $\theta$ is a finite morphism.
Let $v$ be in $V$ and fix an edge $e^{\prime \prime} \in E^{\prime \prime}$ incident to $\psi(v) = \theta(\varphi(v))$.
Then, since $\psi$ is a finite harmonic morphism, we have
\begin{align*}
\operatorname{deg}_v(\psi) &= \sum_{e \in E:\, e \mapsto e^{\prime \prime},\, v \in e} \operatorname{deg}_e(\psi)\\
&= \frac{\sum_{e \in E:\, \varphi(e) \mapsto e^{\prime \prime},\, \varphi(v) \in \varphi(e)} \operatorname{deg}_{\varphi(e)}(\theta)}{|Gv|}\\
&= \frac{\sum_{e^{\prime} \in E^{\prime}:\, e^{\prime} \mapsto e^{\prime \prime},\, \varphi(v) \in e^{\prime}} \operatorname{deg}_{e^{\prime}}(\theta) \cdot |\varphi^{-1}(e^{\prime})|}{|Gv|}\\
&= \frac{|G|}{|Gv|} \cdot \left( \sum_{e^{\prime} \in E^{\prime}:\, e^{\prime} \mapsto e^{\prime \prime},\, \varphi(v) \in e^{\prime}} \operatorname{deg}_{e^{\prime}} (\theta) \right).
\end{align*}
Thus we have
\begin{align*}
\sum_{e^{\prime} \in E^{\prime} :\, e^{\prime} \mapsto e^{\prime \prime},\, \varphi(v) \in e^{\prime}} \operatorname{deg}_{e^{\prime}} (\theta) = \operatorname{deg}_v (\psi) \cdot \frac{|Gv|}{|G|} = \frac{\operatorname{deg}_v (\psi)}{|G_v|}.
\end{align*}
It is independent of the choice of $e^{\prime \prime}$ and since $\operatorname{deg}_{e^{\prime}}(\theta) = \operatorname{deg}_e (\psi)$ with any $e \in E$ such that $\varphi(e) = e^{\prime}$, it is a positive integer.
These mean that $\theta$ is harmonic as $v$ is an element of $V$.
The uniqueness comes from the universal mapping property of quotient topology.
\end{proof}

\begin{rem}
	\label{rem:UMP}
\upshape{
In \cite{JuAe2}, the author gave another definition of Galois coverings on tropical curves.
With that definition, for a map between tropical curves $\varphi : \Gamma \to \Gamma^{\prime}$, if there exists an isometric action of a finite group $G$ on $\Gamma$ such that there exists a finite harmonic morphism of degree one $\theta : \Gamma / G \to \Gamma^{\prime}$ satisfying $\varphi \circ g = \theta \circ \pi_G$ for any $g \in G$, we call $\varphi$ a $G$-Galois covering on $\Gamma^{\prime}$.
In this paper, let us call it a \textit{$G$-preGalois covering} on $\Gamma^{\prime}$ or a \textit{preGalois covering} simply.
There exists a preGalois covering on a tropical curve for which the Galois correspondence and the universal mapping property do not hold.
See Examples \ref{ex:C}, \ref{ex:C2}, and \ref{ex:C1}.
By Examples \ref{ex:C2} and \ref{ex:C1}, we know that for a preGalois covering $\varphi : \Gamma \to \Gamma^{\prime}$, even there exists a model $(G, l)$ for $\Gamma$ such that the greatest common divisor of $\{ \operatorname{deg}_e(\varphi) \,|\, e \in E(G) \}$ is one, $\varphi$ may not be a Galois covering.
}
\end{rem}

\begin{figure}[h]
\begin{center}
\includegraphics[width=0.9\columnwidth]{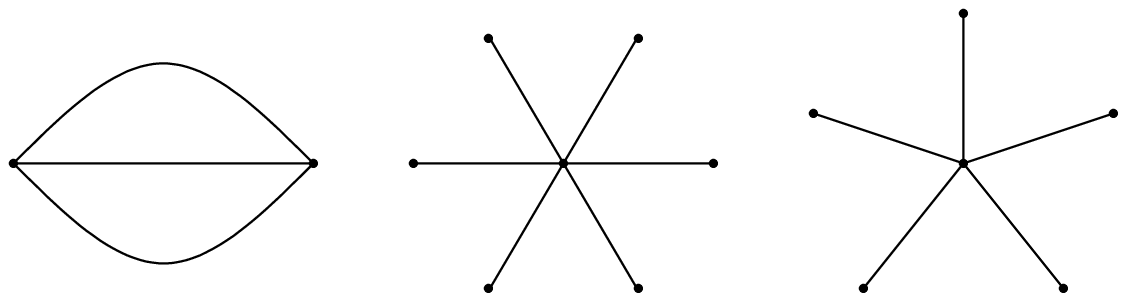}

\vspace{-1mm}
\caption{On each figure, black dots (resp. lines) stand for vertices (resp. edges).}
	\label{figure1}
\end{center}
\end{figure}

\begin{ex}
	\label{ex:C}
\upshape{
Let $G$ be the graph whose set of vertices consists of two vertices and whose set of edges consists of three multiple edges $\{e_1, e_2, e_3 \}$ between them (the left figure of Figure \ref{figure1}).
Let $l$ be the length function such that $l(E(G)) = \{ 1 \}$ and $\Gamma$ the tropical curve obtained from $(G, l)$.
Then, the symmetric group of degree three $\Sigma_3$ isometrically acts on $\Gamma$ in a natural way.
The quotient graph $G / \Sigma_3$ and the length function $E(G / \Sigma_3) \to \boldsymbol{R} \cup \{ \infty \}; [e_i] \mapsto 2$ give the quotient space $\Gamma / \Sigma_3$ a tropical curve structure.
Let $\Gamma^{\prime}$ stand for this quotient tropical curve.
Then, the natural surjection $\pi_{\Sigma_3} : \Gamma \to \Gamma^{\prime}$ is a $\Sigma_3$-preGalois covering on $\Gamma^{\prime}$.
Note that $\Gamma^{\prime}$ is isometric to the closed interval $[0, 2]$.

Let $\sigma$ be the cyclic permutation of multiple edges $(e_1 e_2 e_3)$.
For a finite harmonic morphism $\psi : \Gamma \to \Gamma^{\prime \prime}$, if $\psi \circ \sigma = \psi$, then for any $\beta \in \Sigma_3$, $\psi \circ \beta = \psi$ holds.
Hence, Theorem \ref{thm:Galois correspondence1} does not hold for $\pi_{\Sigma_3}$.

$\pi_{\Sigma_3}$ is not $\Sigma_3$-Galois and does not have a universal mapping property.
In fact, since $\operatorname{deg}_{e_i}(\pi_{\Sigma_3}) = 2$, $\pi_{\Sigma_3}$ is not $\Sigma_3$-Galois.
Since for any $\beta \in \Sigma_3$, $\pi_{\langle \sigma \rangle} \circ \beta = \pi_{\langle \sigma \rangle}$ holds, there exists a unique continuous map $\theta$ satisfying $\pi_{\langle \sigma \rangle} = \theta \circ \pi_{\Sigma_3}$.
Here, $\langle \sigma \rangle$ stands for the subgroup of $G$ generated by $\sigma$.
On the other hand, since $\operatorname{deg}(\pi_{\Sigma_3}) = 6$ and $\operatorname{deg}(\pi_{\langle \sigma \rangle}) = 3$, $\theta$ is not a finite harmonic morphism.
Hence, $\pi_{\Sigma_3}$ does not have a universal mapping property.
}
\end{ex}

\begin{ex}
	\label{ex:C2}
\upshape{
Let $e_1, \ldots, e_6$ be the edges of the star of six edges $S_6$, i.e., the complete bipartite graph $K_{1, 6}$ (the center figure of Figure \ref{figure1}).
Let $l$ be the length function such that $l(E(S_6)) = \{ 1 \}$ and $\Gamma$ be the tropical curve obtained from $(S_6, l)$.
Let $\sigma$ (resp. $\beta$) be the cyclic permutation $(e_1 e_2)$ (resp. $(e_3 e_4 e_5 e_6)$) and $G$ the finite group generated by $\sigma$ and $\beta$.
Then, the quotient graph $S_6 / G$ and the length function $E(S_6 / G) \to \boldsymbol{R} \cup \{\infty\}; [e_1] \mapsto 4; [e_3] \mapsto 2$ give the quotient space $\Gamma / G$ a tropical curve structure. 
Let $\Gamma^{\prime}$ stand for this quotient tropical curve. 
Then, the natural surjection $\pi_{G} : \Gamma \to \Gamma^{\prime}$ is a $G$-preGalois covering.

For a finite harmonic morphism $\psi : \Gamma \to \Gamma^{\prime \prime}$, if $\psi \circ (\sigma \beta^2) = \psi$, then $\psi \circ \beta^2 = \psi$ holds.
Hence, Theorem \ref{thm:Galois correspondence1} does not hold for $\pi_G$.

Let $H$ be the subgroup of $G$ generated by $\beta$.
Then, the natural surjection $\pi_H : \Gamma \to \Gamma / H$ is $H$-preGalois that is not $H$-Galois and does not have a universal mapping property.
In fact, since $\operatorname{deg}_{e_1}(\pi_H) = 4$, $\pi_H$ is not $H$-Galois.
Note that the quotient graph $S_6 / H$ is the tree consisting of three vertices and two edges $[e_1], [e_3]$, and $S_6 / H$ and the length function $E(S_6 / H) \to \boldsymbol{R} \cup \{\infty\}; [e_1] \mapsto 4; [e_3] \mapsto 1$ give the quotient space $\Gamma / H$ a tropical curve structure.
Let $\gamma$ be the cyclic permutation $(e_1 e_2 e_3 e_4 e_5 e_6)$.
Then, the quotient tropical curve $\Gamma / \langle \gamma \rangle$ is isometric to the closed interval $[0, 1]$.
Since $\pi_{\langle \gamma \rangle} \circ \beta = \pi_{\langle \gamma \rangle}$ holds, there exists a unique continuous map $\theta$ satisfying $\pi_{\langle \gamma \rangle} = \theta \circ \pi_H$.
On the other hand, since $\operatorname{deg}(\pi_H) = 4$ and $\operatorname{deg}(\pi_{\langle \gamma \rangle}) = 6$, $\theta$ is not a finite harmonic morphism.
Hence, $\pi_H$ does not have a universal mapping property.
}
\end{ex}

\begin{ex}
	\label{ex:C1}
\upshape{
Let $e_1, \ldots, e_5$ be the edges of the star of five edges $S_5$, i.e., the complete bipartite graph $K_{1, 5}$ (the right figure of Figure \ref{figure1}).
Let $l$ be the length function such that $l(E(S_5)) = \{ 1 \}$ and $\Gamma$ be the tropical curve obtained from $(S_5, l)$.
Let $\sigma$ (resp. $\beta$) be the cyclic permutation $(e_1 e_2)$ (resp. $(e_3 e_4 e_5)$) and $G$ the finite group generated by $\sigma$ and $\beta$.
Then, the quotient graph $S_5 / G$ and the length function $E(S_5 / G) \to \boldsymbol{R} \cup \{\infty\}; [e_1] \mapsto 3; [e_3] \mapsto 2$ give the quotient space $\Gamma / G$ a tropical curve structure. 
Let $\Gamma^{\prime}$ stand for this quotient tropical curve. 
Then, the natural surjection $\pi_{G} : \Gamma \to \Gamma^{\prime}$ is a $G$-preGalois covering.

$\pi_G$ is not $G$-Galois and does not have a universal mapping property.
In fact, since $\operatorname{deg}_{e_i}(\pi_G) \ge 2$, $\pi_G$ is not $G$-Galois.
Let $\Gamma^{\prime \prime}$ be the tropical curve isometric to the closed interval $[0, 1]$.
Note that $\Gamma^{\prime \prime}$ is the quotient tropical curve of $\Gamma$ by the natural action of the group generated by the cyclic permutation $\gamma := (e_1 e_2 e_3 e_4 e_5)$ on $\Gamma$.
Let $\pi_{\langle \gamma \rangle} : \Gamma \to \Gamma^{\prime \prime}$ be the natural surjection.
Since for any $\delta \in G$, $\pi_{\langle \gamma \rangle} \circ \delta = \pi_{\langle \gamma \rangle}$, there exists a unique continuous map $\theta$ satisfying $\pi_{\langle \gamma \rangle} = \theta \circ \pi_{G}$.
On the other hand, since $\operatorname{deg}(\pi_{\langle \gamma \rangle}) = 6$ and $\operatorname{deg}(\pi_{\langle \gamma \rangle}) = 5$, $\theta$ is not a finite harmonic morphism.
Hence, $\pi_{G}$ does not have a universal mapping property.
}
\end{ex}

Note that since Galois coverings are preGalois coverings, for our definition of Galois coverings, all corresponding assertions in \cite{JuAe2} hold and in addition, we need not consider tropical curves with edge-multiplicities.

\end{document}